\documentclass[preprint,12pt]{elsarticle}
\usepackage{amsmath,amssymb,amsthm}

\newtheorem{thr}{Theorem}

\newtheorem{obs}[thr]{Observation}

\theoremstyle{definition}

\newtheorem{prob}[thr]{Problem}

\theoremstyle{remark}

\def\G{{\mathcal G}}
\def\E{{\mathcal E}}
\def\V{{\mathcal V}}
\def\U{{\mathcal U}}
\def\S{{\mathcal S}}
\def\eps{\varepsilon}

\journal{}

\begin{document}

\begin{frontmatter}

\title{On the complexity of failed zero forcing}

\author{Yaroslav Shitov}

\ead{yaroslav-shitov@yandex.ru}

\address{National Research University Higher School of Economics, 20 Myasnitskaya Ulitsa, Moscow 101000, Russia}

\begin{abstract}
Let $G$ be a simple graph whose vertices are partitioned into two subsets, called 'filled' vertices and 'empty' vertices. A vertex $v$ is said to be forced by a filled vertex $u$ if $v$ is a unique empty neighbor of $u$. If we can fill all the vertices of $G$ by repeatedly filling the forced ones, then we call an initial set of filled vertices a forcing set. We discuss the so-called failed forcing number of a graph, which is the largest cardinality of a set which is not forcing. Answering the recent question of Ansill, Jacob, Penzellna, Saavedra, we prove that this quantity is NP-hard to compute. Our proof also works for a related graph invariant which is called the skew failed forcing number.
\end{abstract}


\end{frontmatter}

\section{Introduction}

This writing is devoted to the problem of zero forcing in graphs, which has recently arisen because of the applications in quantum systems theory~\cite{quannet} and minimum rank problems~\cite{AI}. In what follows, we denote by $G$ or $(V,E)$ a finite simple graph with vertex set $V$ and edge set $E$. We assume that the set $V$ is partitioned into two subsets, which we call the \textit{filled} vertices and \textit{empty} vertices.

Let $F\subset V$ be the set of all filled vertices of $G$. We say that an empty vertex $v$ \textit{is forced by} $F$ if there is a filled vertex $u$ whose unique empty neighbor is $v$. (A related concept of \textit{skew forcing} is defined analogously but the vertex $u$ is not required to be filled.) A set $F$ is said to be \textit{(skew) stalled} if there is no empty vertex (skew) forced by $F$. The largest cardinality of a proper (skew) stalled subset of $V$ is called the \textit{(skew) failed forcing number} of $G$. In~\cite{AJPS}, Ansill, Jacob, Penzellna, Saavedra posed a problem to determine the computational complexity of these invariants.

\section{The result}

The goal of this writing is to solve the above mentioned problem. We determine the complexity status of the failed forcing number by proving that the following problems are NP-complete.

\begin{prob}\label{prob11} (FAILED ZERO FORCING.)

\noindent Given: A finite simple graph $G$ and an integer $s$.

\noindent Question: Does $G$ contain a proper stalled subset of cardinality at least $s$?
\end{prob}

\begin{prob}\label{prob12} (FAILED SKEW ZERO FORCING.)

\noindent Given: A finite simple graph $G$ and an integer $s$.

\noindent Question: Does $G$ contain a proper skew stalled subset of cardinality at least $s$?
\end{prob}

We prove our result by constructing a polynomial reduction from INDEPENDENT SET to both of the above problems. Recall that a subset $U\subset V$ is called \textit{independent} if the vertices in $U$ are pairwise non-adjacent.

\begin{prob} (INDEPENDENT SET.)

\noindent Given: A connected simple graph $G$ and an integer $c$.

\noindent Question: Does $G$ contain an independent set of cardinality $c$?
\end{prob}

We recall that the standard formulation of INDEPENDENT SET does not require $G$ to be connected~\cite{Karp}, but the problem remains NP-complete under this restriction. Indeed, we can add to any graph a new vertex adjacent to every other vertex, and this transformation makes the graph connected but does not change the largest cardinality of an independent set.

We proceed with a description of a reduction $G\to\G$ to Problems~\ref{prob11} and~\ref{prob12}. We will assume that $G$ is an instance of INDEPENDENT SET, that is, a connected graph $G$ with the set of vertices $V$ and the set of edges $E$. We will denote by $\V,\E$ the corresponding sets of $\G$. We denote by $n$ the cardinality of $V$, and we set  $$\V=V\cup E^0\cup\ldots\cup E^{2n}\cup\{\eps\},$$ that is, the labels of the vertices of $\G$ are taken from $V$, from the $2n+1$ copies of $E$, and we have one more vertex denoted by $\eps$. We construct the graph $\G$ as follows.

\noindent (1) We subdivide every edge of $G$. That is, we replace every edge $e=\{u,v\}\in E$ by the two edges $\{u,e^0\}$, $\{e^0,v\}$.

\noindent (2) For all $e\in E$, we draw a simple path of length $2n+1$ beginning at $e^0$. In other words, we add vertices $e^1,\ldots,e^{2n}$ and edges $\{e^i,e^{i+1}\}$ for all $i$.

\noindent (3) We add the vertex $\eps$ and edges $\{\eps,e^0\}$ for all $e\in E$.

We need the following theorem to complete the proof of the main result.

\begin{thr}\label{thr11}
If $k$ is the largest cardinality of an independent set of $G$, then the largest proper stalled subset of $\G$ has cardinality $(2n+1)|E|+k$. The same conclusion holds for the largest proper skew stalled subset of $\G$.
\end{thr}

We will give the proof of Theorem~\ref{thr11} in a separate section below. As a corollary of this theorem, we get that $$(G,k)\to (\G, (2n+1)|E|+k)$$ is a polynomial reduction from INDEPENDENT SET to both Problems~\ref{prob11} and~\ref{prob12}. In particular, these problems are NP-complete, so the usual and skew failed forcing numbers are NP-hard to compute.

\section{The proof}

We are going to finalize the paper by proving Theorem~\ref{thr11}. Our first lemma establishes the lower bound on the failed forcing number of $\G$ in terms of $k$. (Here and in the rest of the paper, we denote by $k$ the largest cardinality of an independent set of $G$.)

\begin{obs}\label{obs1}
The graph $\G$ contains a skew stalled subset of cardinality $(2n+1)|E|+k$.
\end{obs}

\begin{proof}
Let $U$ be an independent set of $G$. We need to show that $$\U=U\cup E^0\cup\ldots\cup E^{2n}$$ is a skew stalled subset of $\G$. If this is not the case, then some vertex $x\in\V\setminus\U$ is skew forced by $\U$. This means that either $x=\eps$ or $x\in V\setminus U$, and there is a vertex $y$ for which $x$ is an only neighbor outside $\U$. We treat the two cases separately.

1. The vertex $\eps$ is adjacent only to the vertices $e^0$, so $x=\eps$ implies $y=e^0$ with $e=\{a,b\}\in E$. Since $x$ is an only neighbor of $y$ that lives outside $\U$, the vertices $a,b$ belong to $\U$. So we have $a,b\in U$ and $\{a,b\}\in E$, which is a contradiction because $U$ is an independent set of $G$.

2. Now assume $x\in V\setminus U$. The vertices adjacent to $x$ have the form $e^0$ again, so we get $y=e^0$. This is a contradiction because such a $y$ is adjacent to the vertices $\varepsilon,x\notin\U$.
\end{proof}



\begin{obs}\label{obs3}
If a stalled subset $\S$ of $\G$ contains $e^{i}, e^{i+1}$, for some $e$ and $i$, then $\S$ contains $e^0,\ldots,e^{2n}$ as well.
\end{obs}

\begin{proof}
Assume that the result is not true, which means that $e^t\notin\S$ for some $t$. If $t<i$, then we choose the maximal $\tau<i$ for which $e^\tau\notin\S$. Then $e^\tau$ is the only vertex outside $\S$ which is adjacent to $e^{\tau+1}$. We see that $\S$ forces $e^\tau$,  which is impossible because $\S$ is stalled.

Similarly, if $t>i$, then we choose the minimal ${\tau}>i+1$ for which $e^{\tau}\notin\S$. Then $e^{\tau}$ is the only vertex that lives outside $\S$ and is adjacent to $e^{\tau-1}$. We see that $\S$ forces $e^\tau$ and get a contradiction.
\end{proof}

\begin{obs}\label{obs4}
Let $\S$ be a stalled subset of $\G$ such that $|\S|\geqslant (2n+1)|E|+2$. Then $\mathcal{S}$ contains $e^0,\ldots,e^{2n}$ for all $e\in E$.
\end{obs}

\begin{proof}
Since $|\V|=(2n+1)|E|+n+1$, there are at most $n-1$ vertices of $\G$ that lie outside $\S$. Therefore, there are at least $n+2$ vertices among $e^0,\ldots,e^{2n}$ that belong to $\S$. In particular, there are two consecutive indexes $i,j$ such that $\S$ contains $e^i,e^j$. By Observation~\ref{obs3}, all the vertices $e^0,\ldots,e^{2n}$ belong to $\S$.
\end{proof}

\begin{obs}\label{obs5a}
Let $\S$ be a proper stalled subset of $\G$ such that $|\S|\geqslant (2n+1)|E|+2$ and $\eps\in\S$. Then $\S\cap V$ is a union of several connected components of $G$.
\end{obs}

\begin{proof}
Assume the converse. Then there are vertices $a,b\in V$ such that $a\in\S$, $b\notin\S$, and $e=\{a,b\}\in E$. By the construction of $\G$, the vertex $e^0$ is adjacent to $a,b,\eps,e^1$. We note that $e^0,e^1\in\S$ by Observation~\ref{obs4},  $\eps\in\S$ by the assumption of the lemma, and $a\in\S$, $b\notin\S$ by the above. We see that $b$ is an only neighbor of $e^0$ that lies outside $\S$. So we see that $b$ is forced, which is a contradiction.
\end{proof}

\begin{obs}\label{obs5}
Let $\S$ be a proper stalled subset of $\G$ such that $|\S|\geqslant (2n+1)|E|+2$. Then $\eps\notin\S$.
\end{obs}

\begin{proof}
Assume the converse, which means that $\eps\in\S$. By Observation~\ref{obs4}, we have $E^i\subset\S$ for all $i$. Since the graph $G$ is connected, Observation~\ref{obs5a} implies that either $\S\cap V=\varnothing$ or $\S\cap V=V$. The former condition leads to a contradiction with $|\S|\geqslant (2n+1)|E|+2$, and the latter is impossible because $\S$ is a proper subset.
\end{proof}

\begin{obs}\label{obs6}
Let $\S$ be a proper stalled subset of $\G$ such that $|\S|\geqslant (2n+1)|E|+2$. Then $\mathcal{S}\cap V$ is an independent set of $G$.
\end{obs}

\begin{proof}
Assume the converse. We have that $a,b\in\S\cap V$ and $e=\{a,b\}\in E$. By the construction of $\G$, the vertex $e^0$ is adjacent to $a,b,\eps,e^1$. We note that $a,b\in\S$ by the above, $e^0,e^1\in\S$ by Observation~\ref{obs4}, and $\eps\notin\S$ by Observation~\ref{obs5}. We see that $\eps$ is an only neighbor of $e^0$ that lies outside $\S$. So we see that $\eps$ is forced, which is a contradiction.
\end{proof}

\begin{obs}\label{obs7}
Every proper stalled subset $\S$ of $\G$ has cardinality at most $(2n+1)|E|+k$.
\end{obs}

\begin{proof}
If the result was not true, then Observations~\ref{obs5} and~\ref{obs6} would be applicable. We have $\eps\notin\S$ and $|\S\cap V|\leqslant k$, which means that there are at least $n-k+1$ vertices outside $\S$. The total number of vertices of $\G$ is $(2n+1)|E|+n+1$, so we are done.
\end{proof}

We note that Observations~\ref{obs1} and~\ref{obs7} complete the proof of Theorem~\ref{thr11}. In fact, every skew stalled set is also a stalled set, so the obtained bounds hold for the cardinalities of both usual and skew stalled sets.


\begin{thebibliography}{9}

\bibitem{AI}
AIM Minimum Rank-Special Graphs Work Group, Zero forcing sets and the minimum rank of graphs, Linear Algebra Appl. 428(7) (2008) 1628--1648.

\bibitem{AJPS}
T. Ansill, B. Jacob, J. Penzellna, D. Saavedra, Failed skew zero forcing on a graph, Linear Algebra Appl. 509 (2016) 40--63.

\bibitem{quannet} 
D. Burgarth, V. Giovannetti, Full Control by Locally Induced Relaxation, Phys. Rev. Lett. 99 (2007) 100501.


\bibitem{Karp}
R. Karp, Reducibility Among Combinatorial Problems, Proceedings of the Symposium on the Complexity of Computer Computations (1972) 85--103.
\end{thebibliography}
\end{document}